\documentclass[11pt,reqno]{amsart}
\usepackage{amsmath,amsfonts,amssymb,amscd,amsthm,amsbsy,epsf}
\newtheorem{thm}{Theorem}
\newtheorem*{pitt}{Pitt's Inequality}
\newtheorem*{HLS}{Hardy-Littlewood-Sobolev Inequality}
\newtheorem*{cor}{Corollary}
\newtheorem{lem}{Lemma}

\newtheorem*{ConvoLem}{Convolution Lemma}
\newtheorem*{SymmLem}{Symmetrization Lemma}
\newtheorem*{ReductLem}{Reduction Lemma}
\newtheorem*{SWLem}{Stein-Weiss Lemma}
\newtheorem*{two-pointLem}{Two-Point Symmetrization Lemma}
\theoremstyle{definition}
\newtheorem*{rem}{Remark}
\def\ep{\varepsilon}
\def\H{{\mathcal H}}
\def\S{{\mathcal S}}
\def\complex{{\mathbb C}}
\def\HH{{\mathbb H}}
\def\real{{\mathbb R}}
\def\chy{c_{\text{h-y}}}
\def\Im{\mathop{\text{Im}}\nolimits}
\def\Re{\mathop{\text{Re}}\nolimits}
\begin{document}
\title[Embedding estimates and fractional smoothness]
{Embedding estimates and fractional smoothness}
\author{William Beckner}
\address{Department of Mathematics, The University of Texas at Austin,
1 University Station C1200, Austin TX 78712-0257 USA}
\email{beckner@math.utexas.edu}
\dedicatory{``we shall begin by studying the fractional powers of the Laplacian'' --- Eli Stein}
\begin{abstract}
A short intrinsic proof is given for the Bourgain-Brezis-Mironescu theorem
with an extension for higher-order gradient forms. 
This argument illustrates the role of functional geometry and Fourier analysis 
for obtaining embedding estimates.
New Hausdorff-Young inequalities are obtained for fractional embedding as an extension 
of the classical Aronszajn-Smith formula. 
These results include bilinear fractional embedding as suggested by the Landau collision 
operator in plasma dynamics.
\end{abstract}
\maketitle


Functional forms that characterize smoothness lie at the heart of understanding 
and rigorously describing the many-body interactions that determine the behavior 
of dynamical phenomena. 
Coupled with the establishment of sharp embedding estimates, our overall 
objective is to expand the working framework for $n$-dimensional Fourier analysis
while gaining new insight into uncertainty, 
restriction phenomena and the role of geometric symmetry. 
Efforts to obtain optimal constants bring out new features of exact model 
problems, encoded geometric information and precise lower-order effects.

\section{Bourgain-Brezis-Mironescu theorem}

Let $\Lambda_\alpha = (-\Delta/4\pi^2)^{\alpha/2}$, $\alpha >0$, $1\le p< n/\beta$ and 
$0<\beta <1$; the main purpose here is to establish embedding estimates 
for the Besov norms:
\begin{gather*}
\int_{\real^n\times\real^n} \frac{|f(x) - f(y)|^p}{|x-y|^{n+p\beta}}\,dx\,dy 
\quad \rightsquigarrow\quad
\int_{\real^n\times\real^n} \frac{|(\nabla f)(x) - (\nabla f)(y)|^p}
{|x-y|^{n+p\beta}}\ dx\,dy\\
\noalign{\vskip6pt}
\rightsquigarrow\quad
\int_{\real^n\times\real^n} \frac{|(\Lambda_\alpha f)(x) - (\Lambda_\alpha f)(y)|^p}
{|x-y|^{n+p\beta}}\ dx\,dy
\end{gather*}
The intrinsic character of these norms captures the interplay between dilation 
and translation on $\real^n$. 
Classically such norms measure how differentiability controls function  size and 
determines restriction behavior.

The Bourgain-Brezis-Mironescu theorem corresponds to the estimate for 
$0<\beta <1$ and $1\le p<n/\beta$ (\cite{BBM-00}, \cite{BBM-02}, \cite{MS} and 
p.~521 in \cite{Mazya85}): 
\begin{equation*}
\int_{\real^n\times\real^n}  \frac{|f(x) - f(y)|^p}{|x-y|^{n+p\beta}}\ dx\,dy 
\ge c\bigg( \int_{\real^n} |f|^q\,dx\bigg)^{p/q}\ ,\qquad 
q = \frac{pn}{n-p\beta}
\end{equation*}
For $p=2$, the  best constant $c$ is calculated by the author in \cite{Beckner-Forum}: 
\begin{equation*}
\frac{n-2\beta}{\beta (1-\beta)}\ \pi^{\beta +\frac{n}2} \ 
\frac{\Gamma (2-\beta)}{\Gamma (\frac{n}2 +1-\beta)}\ 
\left[ \frac{\Gamma (\frac{n}2)}{\Gamma (n)}\right]^{2\beta/n}
\end{equation*}
This result extends to include fractional powers of the Laplacian:

\begin{thm}\label{frac-powers1}
For $f \in \S (\real^n)$, $0<\beta <1$ and $1\le p< n/(\alpha+\beta)$ 
\begin{equation}\label{eq-thm1}
\int_{\real^n\times\real^n} 
\frac{|(\Lambda_\alpha f) (x) - (\Lambda_\alpha f)(y)|^p}{|x-y|^{n+p\beta}}\ dx\,dy 
\ge c\bigg( \int_{\real^n} |f|^{q^*} \,dx\bigg)^{p/q^*}\ ,\qquad 
q^* = \frac{pn}{n-p(\alpha+\beta)}
\end{equation}
\end{thm}

\begin{thm}\label{frac-powers2}
For $p= 2< n/(\alpha+\beta)$, the value of $c$ in Theorem~\ref{frac-powers1} 
is given by 
\begin{equation}\label{eq-thm2}
c = \frac2{\beta (1-\beta)} \ \pi^{\beta - \alpha +\frac{n}2}\ 
\frac{\Gamma (2-\beta)}{\Gamma (\frac{n}2 +\beta)}\ 
\frac{\Gamma (\frac{n}2 +\alpha +\beta)}{\Gamma (\frac{n}2 -\alpha-\beta)}\ 
\left[ \frac{\Gamma (\frac{n}2)}{\Gamma (n)}\right]^{2(\alpha+\beta)/n} 
\end{equation}
\end{thm}

\begin{proof}[Proof of Theorem~\ref{frac-powers1}]
This result follows by application of four lemmas utilizing symmetrization, 
Stein-Weiss techniques and the general Hardy-Littlewood-Sobolev inequality. 
Set $g= \Lambda_\alpha f$.
\medskip

\noindent {\sc Step 1:} 
apply the Symmetrization Lemma below to obtain 
\begin{equation}\label{eq-step1}
\int_{\real^n\times\real^n} 
\frac{|g(x) - g(y)|^p}{|x-y|^{n+p\beta}}\ dx\,dy 
\ge \int_{\real^n\times\real^n} 
\frac{|g^* (x) - g^*(y)|^p}{|x-y|^{n+p\beta}}\ dx\,dy
\end{equation}
where $g^*$ is the equimeasurable radial decreasing rearrangement of $|g|$ on 
$\real^n$.
\medskip

\noindent {\sc Step 2:} 
use Lemma~1 below (see Theorem~4.1 in \cite{Beckner-Forum}) to obtain 
\begin{equation}\label{eq-step2}
\int_{\real^n\times\real^n} 
\frac{|g^*(x) - g^*(y)|^p}{|x-y|^{n+p\beta}}\ dx\,dy
\ge D_{p,\beta} \int_{\real^n} |x|^{-p\beta} |g^* (x)|^p\ dx
\end{equation}
with 
$$D_{p,\beta} = \int_{\real^n} \big| 1-|x|^{-\lambda}\big|^p\ |x-\eta |^{-n-p\beta}\ dx$$
for $\lambda = (n-p\beta)/p = n/q$ and $\eta \in S^{n-1}$.
\medskip

\noindent {\sc Step 3:} 
recall that $q = pn/(n-p\beta)$ and let $\sigma (S^n)$ denote the surface area of the 
$n$-dimensional unit sphere. 
Since $g^* (x)$ is non-negative and radial decreasing
$$g^* (x) \le c|x|^{-n/q}\ ,\qquad 
c = \big[ n/\sigma (S^{n-1})\big]^{1/q} \|g^*\|_{L^q(\real^n)}$$
and 
\begin{align}
\int_{\real^n} |x|^{-p\beta} \ |g^*(x)|^p\,dx 
& \ge \left[ \frac{\sigma(S^{n-1})}{n}\right]^{p\beta/n} 
\bigg[ \int_{\real^n}|g^* (x)|^q \,dx\bigg]^{p/q} \notag\\
\noalign{\vskip6pt}
& = \left[ \frac{\sigma (S^{n-1})}{n}\right]^{p\beta/n} 
\bigg[ \int_{\real^n} |\Lambda_\alpha f|^q\,dx \bigg]^{p/q} \label{eq-step3}
\end{align}
\medskip

\noindent {\sc Step 4:} 
now for $0<\alpha < (n-p\beta)/p$ 
$$\bigg[ \int_{\real^n} |\Lambda_\alpha f|^q\,dx\bigg]^{p/q}
\ge c \bigg[ \int_{\real^n}  |f|^{q^*} \,dx\bigg]^{p/q^*}$$
if the Hardy-Littlewood-Sobolev inequality holds for 
$$\Big\| \frac1{|x|^{n-\alpha}} * f\Big\|_{L^{q^*} (\real^n)} 
\le c\|f\|_{L^q (\real^n)}$$
with $c$ a generic constant and $1/q^* = 1/q\ -\ \alpha/n$ which holds for 
$q=np/(n-p\beta)$ and $q^* = np/(n-p(\alpha +\beta))$.

This completes the proof of Theorem 1.
\renewcommand{\qed}{}
\end{proof}

\begin{rem}    
One could look for an alternate proof using Lemma 1 but without first
applying symmetrization. 
The idea would be to use a Stein-Weiss argument 
(see Appendix to \cite{Beckner-PAMS08}) depending on Young's inequality for convolution. 
But the singularity in the convolution kernel is too strong to effectively 
use such a simple approach.
\end{rem}

\begin{SymmLem}
For $f,g$ measurable functions on $\real^n$ with $f^*,g^*$ denoting the equimeasurable 
radial decreasing rearrangement of $|f|, |g|$; $K(x)$ radial decreasing and $p\ge 1$
\begin{equation}\label{eq:symmlem1}
\int_{\real^n\times\real^n} \mkern-36mu 
K(x-y) |f(x) - g(y)|^p\, dx\,dy 
\ge \int_{\real^n\times\real^n} \mkern-36mu 
K(x-y) |f^* (x) -g^* (y)|^p\,dx\,dy
\end{equation}
More generally, for $\rho$ a radial monotone increasing function and $\varphi$ a convex
function on $\real$ satisfying
\begin{itemize}
\item[(1)] $\varphi (t) \ge 0$, $\varphi (0)=0$
\item[(2)] $\varphi$ convex and monotone increasing, $\varphi'' (t) \ge0$
\item[(3)] $t\varphi' (t)$ convex
\end{itemize}
\begin{equation}\label{eq:symmlem2}
\int_{\real^n\times\real^n} \mkern-36mu
K(x-y) \varphi \left[ \frac{|f(x) - g(y)|}{\rho(x-y)}\right]\,dx\,dy 
\ge \int_{\real^n\times\real^n} \mkern-36mu
K(x-y) \varphi \left[ \frac{|f^* (x)-g^*(y)|}{\rho(x-y)} \right]\,dx\,dy
\end{equation}
\end{SymmLem}

The development of this lemma was outlined in the author's paper 
\cite{Beckner-Sobolev}  and motivated by earlier ideas of 
Ahlfors \cite{Ahlfors} (see Lemma~2.2, pages 34--35) perhaps inspired 
by Hardy and Littlewood \cite{HL} but expanding on themes in 
Hardy, Littlewood and P\'olya \cite{HLP}; 
also see Baernstein-Taylor \cite{BT}. 
Details of the proof are given in the Appendix below.

Application of the symmetrization lemma is necessary for Step~3 in the proof of 
Theorem~\ref{frac-powers1}. 
But generally there is no difficulty in making a reduction to radial functions 
for the embedding forms described here. 

\begin{ReductLem}\label{lem:reduct}
Let ``translation'' be a generic representation for any transitive action on 
a manifold for which the volume form is invariant. 
Then for $K$ non-negative and integrable and $1\le p<\infty$
\begin{equation}\label{eq:reductlem}
\int_{M\times M}\mkern-18mu
K(u-v) |f(u)-g(v)|^p\,du\,dv 
\ge \int_M K(u)\,du\,  \big|\, \|f\|_{L^p(M)} - \|g\|_{L^p(M)}\big|^p 
\end{equation} 
\end{ReductLem}

\begin{proof}
\begin{gather*}
\int_{M\times M} \mkern-18mu 
K(u-v)|f(u)-g(v)|^p\,du\,dv
= \int_M K(u) \bigg[ \int_M |f(u+v) - g(v)|^p\,dv\bigg]\,du\\
\noalign{\vskip6pt}
\ge \int_M K(u)\,du \, \big|\,\|f\|_{L^p(M)} - \|g\|_{L^p(M)}\big|^p
\end{gather*}
\end{proof}

By using $\xi \cdot \eta = \xi\cdot (R_\eta \hat e\,) = (R_\eta^{-1}\xi) \cdot \hat e$
on $S^{n-1}$, one obtains from this reduction lemma 
\begin{equation}\label{eq:reductlem2}
\int_{\real^n\times\real^n}\mkern-36mu
K(x-y)|f(x) - g(y)|^p\,dx\,dy 
\ge \int_{\real^n\times\real^n} \mkern-36mu
K(x-y) |F(x) - G(y)|^p\,dx\,dy
\end{equation} 
where 
$$F(x) = \bigg[ \int_{S^{n-1}} |f(|x|\xi)|^p\, d\xi\bigg]^{1/p} \quad ,\quad
G(y) =\bigg[ \int_{S^{n-1}} |g(|y|\xi)|^p\,d  \xi\bigg]^{1/p}$$
with $d\xi$ denoting standard surface measure on the sphere. 

\begin{lem}\label{lem1}
Let $f\in \S(\real^n)$, $0<\beta <1$ and $1\le p< n/\beta$; then 
\begin{equation}\label{eq:lem1}
\int_{\real^n\times\real^n} 
\frac{|f(x) - f(y)|^p}{|x-y|^{n+p\beta}}\,dx\, dy
\ge D_{p,\beta} \int_{\real^n} |x|^{-p\beta} |f(x)|^p\,dx
\end{equation}
\begin{equation*}
D_{p,\beta} = \int_{\real^n} \big|1-|x|^{-\lambda}\big|^p\, |x-\eta|^{-n-p\beta}\,dx
\end{equation*}
for $\lambda = (n-p\beta)/p$ and $\eta \in S^{n-1}$.
\end{lem}

\begin{proof}
The first step is to apply the symmetrization lemma to obtain 
$$\int_{\real^n\times\real^n}  
\frac{|f(x) - f(y)|^p}{|x-y|^{n+p\beta}}\,dx\,dy 
\ge \int_{\real^n\times\real^n} 
\frac{|f^*(x) -f^* (y)|^p}{|x-y|^{n+p\beta}}\,dx\,dy\ .$$
Then set $t= |x|$, $s= |y|$, $h(t) = |x|^{n/p\, -\,\beta } f^* (x)$. 
Then inequality \eqref{eq:lem1} for $f^*$ will be equivalent to the  inequality  on the 
multiplicative group $\real_+$:
\begin{gather}
\int_{\real_+\times\real_+} 
|g(x/t) h(t) - g(t/s) h(s)|^p \psi (s/t)\, \frac{ds}s\, \frac{dt}t \notag\\
\noalign{\vskip6pt}
\ge D_{p,\beta} \int_{\real_+}|h(t)|^p\, \frac{dt}t \label{eq:lem1pf}
\end{gather}
where $g(t) = t^{(n-p\beta)/2p}$,
$$\psi (t) = \int_{S^{n-1}} \Big[ t+ \frac1t - 2\xi_1 \Big]^{-(n+p\beta)/2}\,d\xi$$
and $d\xi$ denotes standard surface measure on $S^{n-1}$. 
Note that $\psi$ is symmetric under inversion. 
Apply the ``triangle inequality lemma'' below using $g\psi^{1/p}$ as the second function 
in the lemma, and one finds that 
\begin{align*}
D_{p,\beta} & = \int_{\real_+} |t^{\lambda/2} - t^{-\lambda/2}|^p \psi (t)\, \frac{dt}t\\
\noalign{\vskip6pt}
& = \int_{\real^n} \big| 1-|x|^{-\lambda} \big|^p \, |x-\eta|^{-n-p\beta}\,dx
\end{align*}
where $\lambda = (n-p\beta)/p$ and $\eta \in S^{n-1}$. 
Since the determination of $D_{p,\beta}$ depends only on application of the 
``triangle inequality'', the constant must be optimal as observed by a suitable variation 
of functions in that inequality. 
To finish the proof of the Lemma~\ref{lem1}, observe the result that 
\begin{gather*}
\int_{\real^n\times\real^n} 
\frac{|f(x)-f(y)|^p}{|x-y|^{n+p\beta}}\,dx\,dy 
\ge \int_{\real^n\times\real^n} 
\frac{|f^*(x) - f^*(y)|^p}{|x-y|^{n+p\beta}}\,dx\,dy\\
\noalign{\vskip6pt}
\ge D_{p,\beta} \int_{\real^n} |x|^{-p\beta} |f^* (x)|^p\,dx 
\ge D_{p,\beta} \int_{\real^n} |x|^{-p\beta} |f(x)|^p\,dx\ .
\end{gather*}
\end{proof}

\begin{lem}[Triangle Inequality]
For $f,g,h\in L^p(\real^m)$, $1\le p <\infty$
\begin{equation}\label{eq:triangle ineq}
\begin{split}
&\int_{\real^m\times\real^m}  
 |g(y-x) f(x) - h(x-y) f(y)|^p\,dx\,dy\\
 \noalign{\vskip6pt}
&\qquad \ge \int_{\real^m} 
 \big|\,|g(y)| - |h(-y)|\,\big|^p\,dy 
\int_{\real^m}    |f(x)|^p\,dx
\end{split}
\end{equation}
\end{lem}

\begin{proof} 
By a change of variables in $y$ and using the triangle inequality for norms in the first variable on 
$\real^n$:
\begin{equation*}
\begin{split}
&\int_{\real^m\times\real^m} |g(y-x) f(x) - h(x-y) f(y)|^p\, dx\,dy\\
\noalign{\vskip6pt}
&\qquad = \int_{\real^m\times\real^m} |g(y) f(x) - h(-y) f(x+y)|^p\,dx\,dy\\
\noalign{\vskip6pt}
&\qquad = \int_{\real^m} \bigg\{ \bigg( \int_{\real^m} |g(y) f(x) - h(-y) f(x+y)|^p\,dx\bigg)^{1/p}
\bigg\}^p\,dy\\
\noalign{\vskip6pt}
&\qquad \ge \int_{\real^m} \Big\{ \Big| \, |g(y)|\, \|f\|_{L^p (\real^m)} 
- |h(-y)|\, \|f\|_{L^p(\real^m)} \Big| \, \Big\}^p\,dy\\
\noalign{\vskip6pt}
&\qquad = \int_{\real^m} \Big|\, |g(y)| - |h(-y)|\, \Big|^p\,dy 
\int_{\real^m} |f(x)|^p\,dx\ .
\end{split}
\end{equation*}
By considering $g,h\ge 0$ and the family $\ep^{m/p} f(\ep x)$, one observes that 
the inequality is optimal.
\end{proof}

A slightly more general form of Lemma~\ref{lem1} can be given that reflects the 
general Stein-Weiss structure for fractional integrals (see the appendix in \cite{Beckner-Sobolev})
and uses an analogous proof (though depending on radial reduction rather than 
symmetrization):

\begin{SWLem} 
Suppose $K$ is a non-negative symmetric kernel defined on 
$\real^n\times \real^n$, continuous on any domain that excludes the 
diagonal, homogeneous of degree $-n-\gamma$, $K(\delta u,\delta v) = 
\delta^{-n-\gamma} K (u,v)$, $0<\gamma <\min (n,p)$, and 
$K(Ru,Rv) = K(u,v)$ for any $R\in SO(n)$. 
Then for $f\in \S (\real^n)$ and $p\ge 1$, 
\begin{align}
&\int_{\real^n\times\real^n} |f(x) - f(y)|^p K(x,y)\,dx\,dy 
\ge D_{p,\gamma} \int_{\real^n} |x|^{-\gamma} |f(x)|^p\,dx
\label{eq:SW-extended}\\
\noalign{\vskip6pt}
&\hskip1truein 
D_{p,\gamma} = \int_{\real^n} |1-|x|^{-\lambda} |^p K(x,\eta)\,dx
\notag
\end{align}
for $\lambda = (n-\gamma)/p$ and $\eta\in S^{n-1}$.
This constant is optimal. 
But note that there is no assumption made that it is finite.
\end{SWLem}

The last lemma required for the proof of Theorem~\ref{frac-powers1} is the 
Hardy-Littlewood-Sobolev inequality. 

\begin{HLS}
For $f\in L^p(\real^n)$, $0<\alpha <n$, $1<p< n/\alpha$ and $1/q = 1/p - \alpha/n$, then 
\begin{equation}\label{eq:HLS1}
\Big\| \frac1{|x|^{n-\alpha}} * f\Big\|_{L^q (\real^n)} 
\le c\| f\|_{L^p(\real^n)}
\end{equation}
and equivalently 
\begin{equation}\label{eq:HLS2}
\|f\|_{L^q(\real^n)} \le c \|\Lambda_\alpha f\|_{L^p(\real^n)}
\end{equation}
where $c$ is a generic constant.
\end{HLS}

A standard proof of this inequality is given by Stein using interpolation 
(see Theorem~1 in chapter 5 in \cite{Stein70}), but  a more intrinsic 
proof using Young's inequality and radial symmetry follows by first applying 
symmetrization and then transferring the inequality to the multiplicative group $\real_+$.
Once simply checks that the function 
$$\psi (t) = t^{n(\frac1p - \frac12 - \frac{\alpha}{2n})} 
\int_{S^{n-1}} \left[ \left( t - \frac1t\right)^2 + 2(1-\xi_1)\right]^{-(n-\alpha)/2}\,d\xi\ ,\qquad t>0$$
is in $L^r(\real_+)$ for $r=n/(n-\alpha)$. 
(See also the discussion of the Stein-Weiss theorem in the appendix to 
\cite{Beckner-PAMS08}.)  
Observe that initial integration on the sphere is necessary for the application
of Young's inequality.

\begin{proof}[Proof of Theorem~\ref{frac-powers2}] 
Apply the classical formula of Aronszajn-Smith (see \cite{Stein70} and Appendix 
to  \cite{Beckner-Forum})
\begin{gather}
\int_{\real^n\times\real^n} 
\frac{|(\Lambda_\alpha f)(x) - (\Lambda_\alpha f)(y)|^2}{|x-y|^{n+2\beta}}\ dx\,dy 
= D_\beta \int_{\real^n} |\xi|^{2\alpha +2\beta} |\hat f(\xi)|^2\,d\xi \label{eq6}\\
\noalign{\vskip6pt}
D_\beta = \frac2{\beta} \ \pi^{\frac{n}2 + 2\beta}\ 
\frac{\Gamma (1-\beta)}{\Gamma (\frac{n}2 +\beta)}
\notag
\end{gather}
and use the dual form of the Hardy-Littlewood-Sobolev inequality due to Lieb:
$$\int_{\real^n} |\xi|^{2\alpha +2\beta} |\hat f(\xi)|^2\,d\xi 
\ge c_{\alpha+\beta} \Big[ \|f\|_{L^{q^*} (\real^n)} \Big]^2$$
with 
$$c_{\alpha+\beta} = \pi^{-(\alpha +\beta)} \ 
\frac{(\frac{n}2 + \alpha+\beta)}{\Gamma (\frac{n}2 - \alpha +\beta)}\ 
\left[ \frac{\Gamma (\frac{n}2)} {\Gamma (n)} \right]^{2(\alpha+\beta)/n}$$
Putting these two statements together gives Theorem~\ref{frac-powers2}.
\renewcommand{\qed}{}
\end{proof}

\begin{cor}
For $f\in \S(\real^n)$, $0<\beta <1$ and $1\le p< n/(1+\beta)$ 
\begin{equation}\label{eq:cor}
\int_{\real^n\times\real^n} 
\frac{|(\nabla f)(x) - (\nabla f)(y)|^p}{|x-y|^{n+p\beta}}\ dx\,dy 
\ge c\bigg( \int_{\real^n} |f|^{q^*} \,dx\bigg)^{p/q^*}\ ,\qquad 
q^* = \frac{pn}{n-p(1+\beta)} 
\end{equation}
\end{cor}

\begin{proof}
Observe that by using the Aronszajn-Smith formula
$$\int_{\real^n\times\real^n} 
\frac{|(\nabla f)(x) - (\nabla f)(y)|^2}{|x-y|^{n+p\beta}}\ dx\,dy 
= \int_{\real^n\times\real^n} 
\frac{|(\Lambda_1 f)(x) - (\Lambda_1 f)(y)|^2}{|x-y|^{n+p\beta}}\ dx\,dy$$
and for $g= |\nabla f|$
$$\int_{\real^n\times\real^n} 
\frac{|(\nabla f)(x) - (\nabla f)(y)|^2}{|x-y|^{n+p\beta}}\ dx\,dy 
\ge \int_{\real^n\times\real^n} 
\frac{|g^* (x) - g^* (y)|^p}{|x-y|^{n+p\beta}}\ dx\,dy$$
Follow the steps in the proof of Theorem~\ref{frac-powers1}  where after 
Step~3 one has for $q= pn/(n-p\beta)$ and using the Sobolev inequality 
\begin{equation}\label{eq:Sobolev}
\bigg[ \int_{\real^n} |\nabla f|^q\,dx\bigg]^{p/q} 
\ge c \, \bigg[ \int_{\real^n} |f|^{q^*} \,dx \bigg]^{p/q^*}
\end{equation} 
for $q^* = pn/(n-p(1+\beta))$. 
This gives the statement of the Corollary.
\end{proof}

\section{Hausdorff-Young inequalities for fractional embedding}

The Aronszajn-Smith formula has proved to be highly useful:
\begin{equation}\label{eq-AS}
\int_{\real^n\times\real^n} 
\frac{|f(x) - f(y)|^2}{|x-y|^{n+2\beta}}\ dx\,dy 
= D_\beta \int_{\real^n} |\xi|^{2\beta} |\hat f(\xi)|^2\,d\xi
\end{equation}
By viewing this formula as an extension of the Plancherel theorem for fractional 
differentiation, one naturally looks to the role of the Hausdorff-Young inequality:
$$\int_{\real^n\times\real^n} 
\frac{|f(x) - f(y)|^p}{|x-y|^{n+p\beta}}\ dx\,dy 
\quad\rightsquigarrow\quad
\int_{\real^n} \left[ |\xi|^\beta| |\hat f(\xi)|\right]^{p'}\,d\xi$$
where $p$ and $p'$ are dual exponents.

\begin{thm}\label{thm-HY}
For $f\in \S (\real^n)$, $0<\beta<1$, $1<p<\infty$ and $1/p\ + \ 1/p' =1$
\begin{align}
\int_{\real^n\times\real^n} 
\frac{|f(x) - f(y)|^p}{|x-y|^{n+p\beta}}\ dx\,dy 
&\ge c\bigg[ \int_{\real^n} \left[ |\xi|^\beta |\hat f(\xi)|\right]^{p'}\,d\xi\bigg]^{p/p'}\ ,\qquad 
1<p\le 2\label{thm3-eq}\\
\noalign{\vskip6pt}
& \le c\bigg[ \int_{\real^n} \left[ |\xi|^\beta |\hat f(\xi)|\right]^{p'}\,d\xi\bigg]^{p/p'}\ ,\qquad 
2\le p<\infty \notag
\end{align}
\end{thm}

\begin{proof} 
Apply the Hausdorff-Young inequality for $1< p \le 2$
\begin{align*}
\int_{\real^n\times\real^n} 
\frac{|f(x) - f(y)|^p}{|x-y|^{n+p\beta}}\ dx\,dy 
& = \int_{\real^n} \frac1{|w|^{n+p\beta}} 
\bigg[ \int_{\real^n} |f(x+w) - f(x)|^p\,dx\bigg]\, dw\\
\noalign{\vskip6pt}
&\ge (\chy)^p 
\int_{\real^n} \frac1{|w|^{n+p\beta}} 
\bigg[ \int_{\real^n} |e^{2\pi iw\cdot\xi} - 1|^{p'}\ |\hat f(\xi)|^{p'}\,d\xi\bigg]^{p/p'}\, dw\\
\noalign{\vskip6pt}
&\ge (\chy)^p 
\Bigg[ \int_{\real^n} |\hat f(\xi)|^{p'} 
\bigg[ \int_{\real^n} | e^{2\pi iw \cdot\xi} -1|^p\ 
\frac1{|w|^{n+p\beta}}\, dw\bigg]^{p'/p}\,d\xi \Bigg]^{p/p'}\\
\noalign{\vskip6pt}
&\ge (\chy)^p 
\int_{\real^n} |e^{2\pi iw\cdot\eta} -1|^p 
\frac1{|w|^{n+p\beta}}\, dw 
\bigg[ \int_{\real^n} \left[ |\xi|^\beta |\hat f(\xi)|\right]^{p'}\, d\xi\bigg]^{p/p'}\\
\noalign{\vskip6pt}
& = c\ \bigg[ \int_{\real^n} \left[ |\xi|^\beta |\hat f(\xi)|\right]^{p'}\,d\xi\bigg]^{p/p'}
\end{align*}
where $\eta \in S^{n-1}$, $\chy$ is the Hausdorff-Young constant and 
Minkowski's inequality for integrals was used to obtain the last inequality above.

For $p\ge2$ one just needs to reverse the string of inequalities:
\begin{align*}
&\int_{\real^n} \frac1{|w|^{n+p\beta}} \ 
\bigg[ \int_{\real^n} |f(x+w) - f(x)|^p\,dx\bigg]\,dw\\
\noalign{\vskip6pt}
&\qquad 
\le  (\chy)^p 
\int_{\real^n} \frac1{|w|^{n+p\beta}}\ 
\bigg[ \int_{\real^n} |e^{2\pi iw\cdot\xi} -1|^{p'} \ 
|\hat f(\xi)|^{p'}\,d\xi \bigg]^{p/p'}\, dw\\
\noalign{\vskip6pt}
&\qquad 
\le  (\chy)^p 
\Bigg[ \int_{\real^n} |\hat f(\xi)|^{p'}\,d\xi\ 
\bigg[ \int_{\real^n} |e^{2\pi iw\cdot\xi} -1|^p \ 
\frac1{|w|^{n+p\beta}} \,dw\bigg]^{p'/p}\,d\xi \Bigg]^{p/p'} \\
\noalign{\vskip6pt}
&\qquad 
=  (\chy)^p 
\int_{\real^n} |e^{2\pi iw\cdot\eta} -1|^p 
\frac1{|w|^{n+p\beta}}\ dw 
\bigg[ \int_{\real^n} |\xi|^\beta |\hat f(\xi)|^{p'}\,d\xi\bigg]^{p/p'} \\
\noalign{\vskip6pt}
&\qquad = c\ \bigg[ \int_{\real^n} \left[ |\xi|^\beta |\hat f(\xi)|\right]^{p'}\,d\xi \bigg]^{p/p'}
\end{align*}
Here 
$$c_{\text{h-y}} = \left[ p^{1/p} / p^{\prime^{1/p'}} \right]^{-n/2}\ .$$
\renewcommand{\qed}{}
\end{proof}

\section{Bilinear fractional embedding} 

Rigorous models for collision dynamics in plasmas suggest that bilinear fractional 
embedding estimates will be useful.
Here the analysis developed above is applied to the forms:
$$\int_{\real^n\times\real^n} 
\frac{|f(x) \nabla g(y) - f(y)\nabla g(x)|}{|x-y|^{n+\lambda}}\,dx\,dy 
\quad\rightsquigarrow\quad 
\int_{\real^n\times\real^n}  
\frac{|f(x)   g(y) - f(y)  g(x)|^p}{|x-y|^{n+\lambda}}\,dx\,dy 
$$
As expected from classical Fourier analysis, these forms can be related to convolution 
formulas --- both on the function side and on the Fourier transform side.

\begin{thm}\label{thm4}
For real-valued $f\in \S(\real^n)$, $0<\lambda <1$ and $\widetilde f (x) = f(-x)$
\begin{equation}\label{eq:thm4}
\begin{split}
&\int_{\real^n\times\real^n} |x-y|^{-n-\lambda} 
|f(x) (\nabla f)(y) - f(y) (\nabla f) (x)|\,dx\,dy\\
\noalign{\vskip6pt}
&\qquad \qquad
\ge 2 \int_{\real^n} |x|^{-n-\lambda} |\nabla (f*\widetilde f\,) (x)|\,dx 
\end{split}
\end{equation}
\end{thm}

\begin{proof}
\begin{equation*}
\begin{split}
&\int_{\real^n\times\real^n} |x-y|^{-n-\lambda} |f(x) (\nabla f)(y) - f(y) (\nabla f)(x)|\,dx\,dy\\
\noalign{\vskip6pt}
&\qquad = \int_{\real^n\times\real^n} |x|^{-n-\lambda} 
  |f(x+y) (\nabla f)(y) - f(y) (\nabla f)(x+y)|\,dx\,dy\\
\noalign{\vskip6pt}
&\qquad \ge \int_{\real^n} |x|^{-n-y} 
\Big| \int_{\real^n} \big[ f(x+y) (\nabla f(y) - f(y) (\nabla f)(x+y)\big]\,dy \Big|\,dx\\
\noalign{\vskip6pt}
&\qquad = \int_{\real^n} |x|^{-n-\lambda} \Big| -2 \int_{\real^n} (\nabla f) (x+y) f(y)\,dy\Big|\, dx\\
\noalign{\vskip6pt}
&\qquad = 2 \int_{\real^n} |x|^{-n-\lambda} \big| \nabla (f*\widetilde f\,) (x)\big|\,dx
\end{split}
\end{equation*}
Observe that 
$$2\nabla (f*\widetilde f\,) (0) 
= \int_{\real^n} \nabla (f^2)\,dy =0$$
so the integrand above is locally integrable at $x=0$ for $0<\lambda <1$.
\renewcommand{\qed}{}
\end{proof}

\begin{thm}\label{thm5}
For $f,g\in \S(\real^n)$, $1<p \le 2$, $1/p + 1/q =1$ and $0<\lambda <q$
\begin{equation}\label{eq:thm5}
\begin{split}
&\int_{\real^n\times\real^n} |x-y|^{-n-\lambda} |f(x) g(y) - f(y) g(x)|^q\, dx\,dy\\
\noalign{\vskip6pt}
&\qquad\qquad 
\le c\bigg[ \int_{\real^n} |H_{\lambda/q} (u)|^p\, du\bigg]^{q/p}
\end{split}
\end{equation}
where 
$$H_{\lambda/q} (u) = \int_{\real^n} |v|^{\lambda/q} \Big| \widehat f \Big(\frac{u+v}2\Big)
\widehat g\Big( \frac{u-v}2\Big)\Big|\,dv\ .$$
\end{thm}

\begin{proof}
\begin{equation*}
\begin{split}
&\int_{\real^n\times\real^n} |x-y|^{-n-\lambda}  |f(x) g(y) - f(y) g(x)|^q\,dx\,dy\\
\noalign{\vskip6pt}
&\qquad 
= \int_{\real^n\times\real^n} |x|^{-n-\lambda} |f(x+y)g(y) - f(y)g(x+y)|^q\,dx\,dy\\
\noalign{\vskip6pt}
&\qquad 
= \int_{\real^n\times\real^n}  |x|^{-n-\lambda} \Big| \int_{\real^n\times\real^n} 
e^{-2\pi iy(\xi+n)} \left[ e^{-2\pi ix\xi} - e^{-2\pi ix\eta} \right] 
\widehat f(\xi) \widehat g (\eta) \,d\xi\,d\eta\Big|^q\,dx\,dy\\
\noalign{\vskip6pt}
&\qquad 
= 2^{-q} \int_{\real^n\times\real^n} |x|^{-n-\lambda} 
\Big| \int_{\real^n} e^{-2\pi iyu} \int_{\real^n} \Delta (xv) \widehat f \Big(\frac{u+v}2\Big) 
\widehat g\Big( \frac{u-v}2\Big)\, du\,dv\Big|^q\,dx\,dy
\end{split}
\end{equation*}
where 
$$\Delta (xv) = e^{\pi ixv} - e^{-\pi ixv}\ .$$
By applying the Hausdorff-Young inequality for $q\ge 2$, one obtains the upper bound
$$\Big(\frac{\chy}2\Big)^q \int_{\real^n} |x|^{-n-\lambda} 
\bigg[ \int_{\real^n} \left| \int_{\real^n} |\Delta (xv)|\,
\Big| \widehat f\, \Big(\frac{u+v}2\Big) \widehat g\Big(\frac{u-v}2\Big)\Big| \,dv\right|^p \, du
\bigg]^{q/p}\,dx\ .$$
The next objective is to interchange the $x$-integration with both integrals for $u$ and $v$
by using Minkowski's inequality for integrals twice.
This gives
\begin{equation*}
\begin{split}
&\le \Big(\frac{\chy}2\Big)^q 
\bigg[ \int_{\real^n} \bigg[ \int_{\real^n} |x|^{-n-\lambda} 
\left| \int_{\real^n} |\Delta (uv)|\, 
\Big|\widehat f\,\Big(\frac{u+v}2\Big)  \widehat g\,\Big(\frac{u-v}2\Big)\Big| \, dv \right|^q\,dx
\bigg]^{p/q} \,du\bigg]^{q/p}\\
\noalign{\vskip6pt}
&\le \Big(\frac{\chy}2\Big)^q
\bigg[ \int_{\real^n} \bigg[ \int_{\real^n} 
\bigg( \int_{\real^n} |x|^{-n-\lambda} |\Delta (xv)|^q\,dx\bigg)^{1/q} 
\Big|\widehat f\,\Big(\frac{u+v}2\Big)  \widehat g\,\Big(\frac{u-v}2\Big)\Big| \, dv 
\bigg]^p\, du\bigg]^{q/p}\\
\noalign{\vskip6pt}
& = c \Big(\frac{\chy}2\Big)^q 
\bigg[ \int_{\real^n} \bigg[ \int_{\real^n} 
|v|^{\lambda/q} 
\Big|\widehat f\,\Big(\frac{u+v}2\Big)  \widehat g\,\Big(\frac{u-v}2\Big)\Big| \, dv 
\bigg]^p\, du\bigg]^{q/p}\\
\noalign{\vskip6pt}
& = c \Big(\frac{\chy}2\Big)^q
\bigg[ \int_{\real^n} |H_{\lambda/q} (u)|^p \,du\bigg]^{q/p}
\end{split}
\end{equation*}
where $\chy$ is the Hausdorff-Young constant and 
$$c = \int_{\real^n} |x|^{-n-\lambda} |\Delta (x\cdot\eta)|^q\,dx\ ,\qquad 
\eta \in S^{n-1}$$
where $0<\lambda <q$.
\end{proof}

More careful analysis of the bilinear embedding integral for $p=2$ will give the analogue 
of the classical Aronszajn-Smith identity.

\begin{thm}\label{thm6}
For $f,g\in \S(\real^n)$ and $0<\lambda <2$
\begin{equation}\label{eq:thm6}
\begin{split}
&\int_{\real^n\times\real^n} |x-y|^{-n-\lambda} 
|f(x) g(y)- f(y)g(x)|^2 \,dx\,dy\\
\noalign{\vskip6pt}
&\quad = c\int_{\real^n\times\real^n\times\real^n}\mkern-60mu
|v_1 + v_2|^\lambda \widehat f\Big(\frac{u+v_1}2\Big) \widehat g \Big(\frac{u-v_1}2\Big) 
\left[ \skew6\bar{\widehat f} \Big( \frac{u+v_2}2\Big) 
\bar{\widehat g}\Big( \frac{u-v_2}2\Big) 
- \skew6\bar{\widehat f} \Big( \frac{u-v_2}2\Big) 
\bar{\widehat g} \Big(\frac{u+v_2}2\Big)\right]\, du\,dv_1\,dv_2
\end{split}
\end{equation}
where
$$c = \Big( \frac{\pi}2\Big)^\lambda\ \frac{\pi^{n/2}}{\lambda}\ 
\frac{\Gamma (1-\frac{\lambda}2)}{\Gamma (\frac{n+\lambda}2)}\ .$$
\end{thm}

\noindent 
Here the value of the integral 
$$\int_{\real^n} \frac1{|w|^{n+\lambda}} (1-\cos w\cdot\eta)\,dw 
= \frac{2^{1-\lambda} \pi^{n/2}}{\lambda}\ 
\frac{\Gamma(1-\frac{\lambda}2)}{\Gamma (\frac{n+\lambda}2)}$$
was calculated in the appendix of  \cite{Beckner-Forum}. 

Functional product decomposition coupled with fractional embedding norms modeled 
after plasma collision dynamics can be used to determine new inequalities of 
Bourgain-Brezis-Mironceau type.

\begin{thm}\label{thm7}
For $f,g\in \S(\real^n)$, $0<\beta <1$ and$1\le p<n/\beta$
\begin{equation}\label{eq:thm7}
\begin{split} 
&\int_{\real^n\times\real^n\times\real^n\times\real^n} 
\left[ (x-u)^2 + (y-v)^2\right]^{-(2n+p\beta)/2} 
|f(x) g(y) - f(u) g(v)|^p\, dx\,dy\,du\,dv\\
\noalign{\vskip6pt}
&\qquad \ge c \left\{ \begin{array}{l}
\Big[ \|f\|_{L^p (\real^n)} \|g\|_{L^q(\real^n)}\Big]^p\\
\noalign{\vskip6pt}
\Big[ \|g\|_{L^p(\real^n)} \|f\|_{L^q(\real^n)}\Big]^p
\end{array} \right.
\end{split}
\end{equation} 
where $q= pn /(n-p\beta)$. 
For $p=2$, the sharp value of the constant is given by 
$$c = \frac{2\pi^{\beta+n}}{\beta}\ 
\frac{\Gamma (1-\beta)}{\Gamma (\frac{n}2 -\beta)}\ 
\frac{\Gamma (\frac{n}2 +\beta)}{\Gamma (n+\beta)}\ 
\left[ \frac{\Gamma (\frac{n}2)}{\Gamma (n)} \right]^{2\beta/n}$$
\end{thm}

\begin{proof}
By applying the Symmetrization Lemma separately in $(x,u)$ and $(y,v)$, the integral 
is reduced by using the equimeasurable radial decreasing rearrangements of $f$ and $g$.
Then apply the triangle inequality lemma for the function $f$ to obtain the lower bound:
\begin{gather*}
\int_{\real^n} |f|^p\,dx \int_{\real^n} (1+|x|^2)^{-(2n+p\beta)/2}\,dx 
\int_{\real^n\times\real^n} |y-v|^{-n-p\beta} 
|g^* (y) - g^* (v)|^p\,dy\,dv\\
\noalign{\vskip6pt}
\ge c \int_{\real^n} |f|^p\,dx \bigg( \int_{\real^n} |g|^q\,dx\bigg)^{p/q}\ ,\qquad 
q = \frac{pn}{n-p\beta}\ .
\end{gather*}
Since the fractional embedding norm is symmetric in the role of the functions $f$ and $g$, 
this gives equation~\eqref{eq:thm7}.
Observe that in fact the finiteness of the embedding norm implies that 
$f,g\in L^p (\real^n)\cap L^q (\real^n)$. 
The sharp value of the constant for $p=2$ can be obtained using the Aronszajn-Smith 
formula for $\real^{2n}$ and then applying Theorem~\ref{frac-powers2} for $\alpha =0$.
\end{proof}

While this analysis is instructive and perhaps useful in understanding collision dynamics, 
it does reflect that restriction to product functions will not accurately show the appropriate 
control by multilinear embedding (see discussion on page 185 in \cite{Beckner2012}).

\section{Fractional smoothing on the Heisenberg group}

The Heisenberg group is a natural setting to study problems with mixed homogeneity.
As a nilpotent Lie group, it gives the simplest extension of $\real^n$ where homogeneity is 
broken but a substantive part of the Euclidean translation group action is retained. 
Moreover the group possesses an $SL(2,\real)$ symmetry associated with its dilation 
structure which brings into play the non-unimodular group associated with the hyperbolic plane.
The complexity of the symmetry structure makes some questions arising from Euclidean 
analysis hard while other issues become relatively simple to sort out with multiple approaches.
This framework highlights the intrinsic tension between viewing a nilpotent semisimple 
group arising from complex geometry versus a manifold with non-positive curvature 
and a non-unimodular group action as the most characteristic extension of $\real^n$.

The Heisenberg group $\H_n$ is realized as the boundary of the Siegel upper half-space
in $\complex^{n+1}$, $D = \int z\in \complex^{n+1} : \Im z_{n+1} > |z_1|^2 + \cdots + 
|z_n|^2\}$. 
Then 
$$\H_n = \big\{ w= (z,t) : z\in \complex^n ,\, t\in \real\}$$
with the group action 
$$ww' = (z,t) (z',t')  = (z+z',\, t+t' +2\Im z\bar z')$$
and Haar measure on the group is given by 
$$dw = dz\,d\bar z\,dt = 4^n\,dx\,dy\,dt$$
where $z = x+iy \in \complex''$ and $t\in \real$. 
The natural metric is
$$d(w,w') = d\big( (z,t),(z',t')\big) = d(w^{'-1}w,\widehat 0\,)$$
with 
$$d(w,\widehat 0\,) = d\big( (z,t),(0,0)\big) = \big|\, |z|^2 + it\big|^{1/2} 
= \big|\, |z|^4 + t^2\big|^{1/4} = |w| \ .$$
To illustrate the relation with fractional smoothness, two examples expanding on Besov 
norms and Stein-Weiss fractional integrals are developed here.
But first as a technical tool to facilitate application relative to problems with mixed 
homogeneity, the ``triangle inequality'' lemma is rephrased as a convolution inequality.

\begin{lem}\label{lem3}
For $k\ge 0$ and $f,g,k$ satisfying suitable integrability conditions 
\begin{equation}\label{eq:lem3}
\begin{split}
&\int k(u,v,t-s) |f(u,t) - g(v,s)|^p\,du\,dv\,ds\,dt\\
\noalign{\vskip6pt}
&\qquad \ge\int \bigg[ \int k(u,v,t)\,dt\bigg] \, |F(u) - G(v)|^p\,du\,dv
\end{split}
\end{equation}
where 
$$F(u) = \bigg( \int |f(u,t)|^p \,dt\bigg)^{1/p}\quad ,\quad 
G(v) = \bigg( \int |g(v,t)|^p\,dt\bigg)^{1/p}\ .$$
\end{lem}

\begin{proof}
Observe that 
\begin{equation*}
\begin{split}
&\int k(u,v,t-s) |f(u,t) - g(v,s)|^p\, du\,dv\,dt\,ds\\
\noalign{\vskip6pt}
&\qquad = \int k(u,v,t) \bigg[ \int |f(u,t+s) - g(v,s)|^p\,ds\bigg] \,du\, dv\,dt
\end{split}
\end{equation*}
and apply the ``triangle inequality'' for the $s$-integration; that is 
$$\int |f(u,t+s) - g(v,s)|^p\,ds \ge |F(u) - G(v)|^p\ .$$
\end{proof}

\begin{thm}\label{thm8}
Let $f\in \S (\H_n)$, $0<\beta <1$ and $1\le p< 2n/\beta$; then 
\begin{gather}
\int_{\H_n\times\H_n} 
\frac{|f(w) - f(w')|^p}{[d(w,w')]^{2n+2+p\beta}}\,dw\,dw' 
\ge F_{p,\beta} \int_{\H_n} |z|^{-p\beta} |f|^p\,dw\label{eq:thm8}\\
\noalign{\vskip6pt}
F_{p,\beta} = \frac{4^n \sqrt{\pi}\, \Gamma [\frac{2n+p\beta}4]}{\Gamma [\frac{2n+2+p\beta}4]}
\int_{\real^{2n}} \big| 1-|x|^{-\lambda}\big|^p\, |x-\eta|^{2n-p\beta}\,dx\notag
\end{gather}
for $\lambda = (2n-p\beta)/p$ and $\eta \in S^{2n-1}$.
\end{thm}

\begin{proof}
Apply the ``triangle inequality'' to the $t,t'$ integrations
\begin{equation*}
\begin{split}
&\int_{\H_n\times\H_n} \frac{|f(w)-f(w')|^p}{[d(w,w')]^{2n+2+p\beta}}\, dw\,dw'\\
\noalign{\vskip6pt}
&\qquad \ge \int_{\complex^n\times\complex^n} 
\bigg[ \int_{\real} \big[ |z-z'|^4 + t^2\big]^{-(2n+2+p\beta)/4}\,dt \bigg] 
|h(z) - h(z')|^p\,dz\,d\bar z\, dz'\, d\bar z'\\
\noalign{\vskip6pt}
&\qquad  = \frac{\sqrt{\pi}\,\Gamma [\frac{2n+p\beta}4]}{\Gamma [\frac{2n+2+p\beta}4]}
\int_{\complex^n\times\complex^n}
\frac{|h(z)- h(z')|^p}{|z-z'|^{2n+p\beta}}\, dz\,d\bar z\, dz'\, d\bar z'\\
\noalign{\vskip6pt}
&\qquad \ge \frac{4^n\sqrt{\pi}\,\Gamma [\frac{2n+p\beta}4]}{\Gamma[\frac{2n+2+p\beta}4]}
\int_{\real^{2n}} \big| 1-|x|^{-\lambda}\big|^p |x-\eta |^{2n-p\beta}\,dx 
\int_{\complex^n} |z|^{-p\beta}\, |h|^p \,dz\,d\bar z
\end{split}
\end{equation*}
for $\lambda = (2n-p\beta)/p$ and $\eta \in S^{2n-1}$ and using Lemma~\ref{lem1} with 
$$h(z) = \bigg( \int_{\real} |f(z,t)|^p\, dt\bigg)^{1/p}$$
to obtain inequality~\eqref{eq:thm8} above.
\end{proof}

By combining methods taken from earlier papers (see Theorem~3 in \cite{Beckner-97}, 
Theorem~5 in \cite{Beckner-PAMS08}), one can obtain sharp estimates for 
Stein-Weiss fractional integrals on the Heisenberg group as maps from 
$L^p (\H_n)$ to $L^p (\H_n)$ for $1<p<\infty$ (see also \cite{HLZ} where different 
arguments are developed for the analysis of Stein-Weiss integrals). 

\begin{thm}\label{thm9} 
For $f\in L^p (\H_n)$, $1<p<\infty$, $w= (z,t) \in \H_n$, $0<\lambda < 2n+2$, 
$\alpha < 2n/p$, $\beta < 2n/p'$, $1/p + 1/p' =1$ and $2n+2 = \lambda +\alpha+\beta$
\begin{gather}
\Big\|\, |z|^{-\alpha} \Big( |w|^{-\lambda} * \big( |z|^{-\beta} f\big)\Big)\Big\|_{L^p(\H_n)}
\le D_{\alpha,\beta} \|f\|_{L^p(\H_n)}\label{eq:thm9}\\
\noalign{\vskip6pt}
D_{\alpha,\beta} = (4\pi^2)^n\, \sqrt{\pi}
\left[ \frac{\Gamma [\frac{2n-\alpha-\beta}4] \Gamma [\frac{\alpha+\beta}2] 
\Gamma [\frac{n}p - \frac{\alpha}2] \Gamma [\frac{n}{p'} - \frac{\beta}2] }
{\Gamma [\frac{2n+2-\alpha-\beta}4] \Gamma [\frac{2n-\alpha-\beta}2 ]
\Gamma [\frac{n}{p'} + \frac{\alpha}2] \Gamma [\frac{n}{p} + \frac{\beta}2]}  \right]  \notag
\end{gather}
\end{thm}

\begin{proof} 
Apply the sharp $L^1$ Young's inequality for convolution in the $t$-variable which will 
then result in a reduction to the corresponding Stein-Weiss fractional integral on 
$\real^{2n}$ (see Theorem~2 in \cite{Beckner-PAMS08}). 
Here $|w|$ denotes $(|z|^4 + t^2)^{1/4}$ which defines the metric on $\H_n$, and $*$ 
denotes convolution on the group in equation \eqref{eq:thm9}.
$$\Big\|\, |z|^{-\alpha} \Big( |w|^{-\lambda} * \big( |z|^{-\beta}f\big)\Big)\Big\|_{L^p(\H_n)}
\le \Big\|\, |z|^{-\alpha}\Big( J * \big( |z|^{-\beta}h\big)\Big)\Big\|_{L^p(\complex^n)}$$
where 
\begin{gather*}
h(z) = \bigg[ \int_{\real} |f(z,t)|^p\,dt\bigg]^{1/p}\quad ,\quad 
\|h\|_{L^p(\complex^n)} = \|f\|_{L^p(\H_n)}\ ,\\
\noalign{\vskip6pt}
J(z) = \int_{\real} \big( |z|^4 + t^2\big)^{-\lambda/4}\,dt 
= |z|^{-\lambda +2} \int_{-\infty}^\infty (1+t^2)^{-\lambda/4}\,dt\\
\noalign{\vskip6pt}
= \frac{\sqrt{\pi}\, \Gamma (\frac{2n-\alpha-\beta}4)} {\Gamma (\frac{2n+2-\alpha-\beta}4)} \ 
|z|^{-\lambda+2}
\end{gather*}
then
\begin{equation*}
\begin{split}
&\Big\| \, | z|^{-\alpha} \Big( J * \big( |z|^{-\beta} h\big)\Big)\Big\|_{L^p (\complex^n)}\\
\noalign{\vskip6pt}
&\qquad 
= 4^{n(1+\frac1p)} \sqrt{\pi}\, 
\frac{\Gamma [\frac{2n-\alpha-\beta}4]} {\Gamma [\frac{2n+2-\alpha-\beta}4] }\ 
\Big\|\, |x|^{-\alpha} \Big( |x|^{-\alpha} * \big( |x|^{-\beta} h\big)\Big)\Big\|_{L^p(\real^{2n})}\\
\noalign{\vskip6pt}
&\qquad 
\le (4\pi^2)^n \sqrt{\pi}\, 
\left[ \frac{\Gamma [\frac{2n-\alpha-\beta}4] \Gamma [\frac{\alpha+\beta}2] 
\Gamma [\frac{n}p - \frac{\alpha}2] \Gamma [\frac{n}{p'} - \frac{\beta}2] }
{\Gamma [\frac{2n+2-\alpha-\beta}4] \Gamma [\frac{2n-\alpha-\beta}2 ]
\Gamma [\frac{n}{p'} + \frac{\alpha}2] \Gamma [\frac{n}{p} + \frac{\beta}2]}  \right]
\| h\|_{L^p(\complex^n)}\\
\noalign{\vskip6pt}
&\qquad 
= D_{\alpha,\beta} \|f\|_{L^p (\H_n)}
\end{split}
\end{equation*}
Notice that as demonstrated by the above calculation the hypothesis requires 
that $\lambda >2$ so that clearly the integral   above for $J$ is well-defined. 
\end{proof}

Analysis on the Heisenberg group reflects the characteristic property that the 
transitive group action ``breaks'' the closely coupled Euclidean translation and 
homogeneity in a way that does not allow intrinsic symmetries to clearly determine precise 
estimates. 
The complexity of the intertwining of lower-order invariance on the Heisenberg group, 
perhaps reflecting the influence of complex geometry, does not facilitate decoupling 
the contrasting symmetries. 
Partial discussion of this behavior is given in arguments contained in the author's 
paper \cite{Beckner-95Stein}. 
Proofs for the theorems above use integrability for the kernel to simply remove the 
non-Euclidean part of the convolution. 
But here a reductioin to functions radial in the $z$ variable will provide alternate 
proofs that utilize the underlying two-dimensional dilation symmetry on $\H_n$ that 
corresponds to the embedded action of $SL(2,\real)$ and hyperbolic space $\HH^2$.
Start with the metric on the Heisenberg group,
\begin{align*}
|w^{'-1} w| & = \Big[ |z-z'|^4 + |t-t' -2 \Im z\bar z{}'|^2 \Big]^{1/4}\\
& = \Big[ \big|\, |z|^2 + |z'|^2 - 2\Re z\bar z{}'\big|^2 
+ |t-t' - 2\Im z\bar z{}'|^2\Big]^{1/4}\ ;
\end{align*}
without too much overlap, change notation so that $y = |z|$ and $u_z \in SU(n)$ 
with $u_z (1,0) = z/|z|$, $\langle u\rangle = u_{1,1}$; then 
\begin{align*}
|w'{}^{-1} w| 
& = \Big[ \big| y+y' - 2\sqrt{yy'}\, \Re \langle U^{'-1} U\rangle\big|^2 
+ \big| t-t' - 2\sqrt{yy'}\, \Im \langle U^{'-1} U\rangle\big|^2 \Big]^{1/4}\\
& = (4yy')^{1/4} \left[ \Big| \frac{y+y'}{2\sqrt{yy'}} - \Re\zeta\Big|^2 
+ \Big| \frac{t-t'}{2\sqrt{yy'}} - \Im \zeta\Big|^2 \right]^{1/4}\ ,\qquad 
\zeta = \langle U^{'-1} U\rangle\\
& = (4yy')^{1/4} \Big[ \rho^2 -2\rho |\zeta | \cos (\theta -\varphi) + |\zeta |^2\Big]^{1/4}
\end{align*}
where 
$$\rho^2 = \left( \frac{y+y'}{2\sqrt{yy'}}\right)^2 
+ \left( \frac{t-t'}{2\sqrt{yy'}}\right)^2 = 1+\delta^2$$
with $\delta$ denoting the Poincar\'e metric on $\HH^2$
$$\delta = \frac1{2\sqrt{yy'}} \, \sqrt{(t-t')^2 + (y-y')^2} 
= \text{dist} \big[ (t,y), (t',y')\big]$$
and 
$$\tan \theta = \frac{t-t'}{y+y'} \ ,\quad \frac{-\pi}2 < \theta < \frac{\pi}2\ ;\quad 
\zeta = |\zeta |e^{i\varphi}\ ,\quad 0\le \varphi < 2\pi\ .$$
the connected rotations defined by the angles $\theta$ and $\varphi$, one depending on 
the dilation variables and the other on the variables coming from the complex boundary 
of the ball, create the difficulty for reduction to lower-dimensional submanifolds. 
But the problems here allow an integration to remove this interaction so that the kernel 
depends only on the Poincar\'e metric as a decreasing function of $\rho$ or $\delta$:
$$\psi_\lambda (\rho) = \int_{\partial B_n} |\rho - \zeta_1|^{-\lambda}\, d\zeta$$
where $\partial B_n = \{ \zeta \in \complex^n : |\zeta|^2 = \sum |\zeta_k|^2 =1\}$
and $d\zeta$ denotes normalized surface measure on the boundary. 
This representation allows not only development of an alternate framework for the 
proofs of Theorems~\ref{thm8} and \ref{thm9} but gives new insight on the role of 
embedding and convolution estimates on hyperbolic space $\HH^n$ and non-unimodular 
groups with non-positive curvature. 

\begin{proof}[Alternate proof (Theorem~\ref{thm9})] 
The weighted fractional integral on the Heisenberg group given by equation~\eqref{eq:thm9}
is equivalent to the inequality on $\HH^2$
\begin{gather}
\Big| \int_{\HH^2\times\HH^2} G(v) K(v^{'-1} v) F (v') \,d\nu \, d\nu'\Big|
\le C_{\alpha,\beta,p} \|F\|_{L^p (\HH^2)} \|G\|_{L^{p'} (\HH^2)} \label{eq:alternate}\\
\noalign{\vskip6pt}
\|F * K\|_{L^p(\HH^2)} \le C_{\alpha,\beta,p} \|F\|_{L^p(\HH^2)} \notag\\
\noalign{\vskip6pt}
C_{\alpha,\beta,p} = \|\Delta^{-1/p'} K\|_{L^1 (\HH^2)}\notag
\end{gather}
where
$$K(v) = y^\sigma \psi_\lambda \big[ \rho (v,\hat 0\,)\big]$$
with $\sigma = (n+1)( 1/p - 1/2) -\  \alpha/4\ +\  \beta/4$, 
$\rho^2 = 1+\delta^2 (v,\hat 0\,)$, $\lambda = 2n+2 -\alpha-\beta$, and $\Delta$ 
denotes the modular function on $\HH^2$.
Here Young's inequality for non-unimodular groups is applied. 
As for the Euclidean case, the $L^1$ estimate is sharp.
\end{proof}

\begin{proof}[Alternate proof (Theorem~\ref{thm8})] 
First apply the triangle inequality in the angular variables on $S^{2n-1}$ to reduce the 
problem to functions radial in $z$ and the integration falls over the variables 
$(|z|,t)$.
\begin{equation*}
\begin{split}
&
\int_{\H_n\times\H_n} 
\frac{|f(w) - f(w')|^p}{[d(w,w')]^{2n+2+p\beta}}\, dw\, dw'  \\
\noalign{\vskip6pt}
&\qquad \ge \int_{\H_n\times \H_n} \big| F(|z|,t) - F(|z'|,t') \big|^p 
\bigg[ \int \big[ d(w,w')\big]^{-(2n+2+p\beta)}\,d\hat\xi\bigg]\,dw\, dw'
\end{split}
\end{equation*}
where $dw = dz\,d\bar z\, dt = 4^n r^{2n-1}\, dr\, d\xi\,dt$ with $d\hat\xi$ denoting 
normalized surface measure on $S^{2n-1}$ and 
$$F(|z|,t) = \bigg[ \int_{S^{2n-1}} \big| f(w)\big|\, d\hat\xi\bigg]\ .$$
To repose this problem on the hyperbolic plane $\HH^2$, set $y= |z|^2$.
Then 
\begin{equation*}
\begin{split}
&\int_{\H_n\times\H_n} \big| F(|z|,t) - F(|z'| ,t')\big|^p
\bigg[ \int_{S^{2n-1}} \big[ d(w,w')\big]^{-(2n+2+p\beta)}\,d\hat\xi\bigg]\,dw\,dw'\\
\noalign{\vskip6pt}
&\qquad 
= (4\pi)^{2n}   2^{-(n+1+p\beta/2)}    
\Big/ \big[ \Gamma (n)\big]^2 
\int_{\HH^2\times\HH^2} \psi_\lambda (\rho) \Big| h(v) (y'/y)^{\sigma/2} 
- h(v') (y/y')^{\sigma/2}\Big|^p \,d\nu\, d\nu'
\end{split}
\end{equation*}
where $d\nu = y^{-2}\,dx\,dy$ is Haar measure on $\HH^2$, 
$\sigma = \frac{n+1}p - \frac{\beta}2$ and 
$$\psi_\lambda (\rho) = \int_{\partial B_n} |\rho -\zeta_1|^{-\lambda} \,d\zeta$$
where $\lambda = 2n +2+p\beta$ and $\rho = \sqrt{1+\delta^2}$, 
$\delta = d(v,v')$ is the Poincar\'e metric on $\HH^2$ with $v = (x,y)\in \HH^2$. 
$h(v) = F(\sqrt{y},x) y^{-(n+1-p\beta)/p}$ and the one-dimensional variable $t$ 
has be re-labeled as $x$.
Further setting
$$g(v) = y^{\sigma/2} \Big[ \psi_\lambda (\rho)\Big]^{1/p}$$
where here $\rho = \sqrt{1+\text{dist}(v,\hat 0)^2}$ with $\hat 0 = (0,1)$ being the 
origin in $\HH^2$.

Then
\begin{equation}\label{eq:pf-thm8}
\begin{split}
&\int_{\HH^2\times\HH^2} \psi_\lambda (\rho) \Big| h(v) (y'/y)^{\sigma/2} - h(v') 
(y/y')^{\sigma/2}\Big|^p\,d\nu\,d\nu'\\
\noalign{\vskip6pt}
&\qquad = \int_{\HH^2\times\HH^2} \Big| g(v^{-1} v') h(v) - g(v'{}^{-1} v) h(v')\Big|^p\, 
d\nu\, d\nu'\\
\noalign{\vskip6pt}
&\qquad\qquad 
\ge C_{p,\beta} \int_{\HH^2} |h|^p\,d\nu
\end{split}
\end{equation}
by applying the triangle inequality for non-unimodular groups (see equation~(4.3) 
on page~189 in \cite{Beckner-Forum}) where 
$$C_{p,\beta} =  \int_{\HH^2} \Big| \, |g(v)| - \Delta (v)^{-1/p} |g(v^{-1})|\, \Big|^p\, d\nu$$
with $\Delta$ denoting the modular function which is $1/y$ on $\HH^2$.
Checking all the changes in variables results in the constant obtained in 
equation~\eqref{eq:thm8} above.
\end{proof}

For clarity, the basic inequalities for a Lie group are listed in the following lemma:

\begin{ConvoLem}
Let $G$ be a locally compact group with left-invariant Haar measure denoted by $m$.
For $1\le p\le \infty$
\begin{gather*}
\| f* g\|_{L^p(G)} \le \|f\|_{L^p(G)} \|\Delta^{-1/p'} g\|_{L^1 (G)}\\
\noalign{\vskip6pt}
 \| f*g\|_{L^p(G)} \le \|f\|_{L^1(G)} \|  g\|_{L^p (G)}\\
\noalign{\vskip6pt}
\| f*g\|_{L^r(G)} \le \|f\|_{L^p(G)} \|\Delta^{-1/p'} g\|_{L^q (G)}
\end{gather*}
where $\Delta$ denotes the modular function defined by $m(Ey) = \Delta (y) m(E)$,
$1/p + 1/p' = 1$ and $1/r = 1/p + 1/q-1$. 
For $f,g,h \in L^p (G)$, $1\le p<\infty$
\begin{align*}
&\int_{G\times G}  |g(x^{-1} y) f(x) - h(y^{-1} x) f(y) |^p \, dm\, dm\\
\noalign{\vskip6pt}
&\qquad \ge \int_G \Big| \,  | g(y) | - \Delta (y)^{-1/p}  | h(y^{-1}) |\, \Big|^p\,dm 
\int_G \big| f(x)\big|^p\,dm
\end{align*}
The first, second and fourth inequalities are optimal.
\end{ConvoLem}

\section{Pitt's inequality}

The Hausdorff-Young estimates obtained above allow one to give some reasonable (though still
not optimal) constants for Pitt's inequality on the line of duality.
{From} the Appendix in \cite{Beckner-PAMS08}:

\begin{pitt}
For $f\in\S(\real^n)$, $1<p\le q< \infty$, $0<\alpha < n/q$, $0<\beta <n/p'$ and $n\ge 2$
\begin{equation}\label{eq:pitt}
\bigg[ \int_{\real^n} \Big|\, |x|^{-\alpha} \widehat f \,\Big|^q\,dx\bigg]^{1/q}
\le A \bigg[ \int_{\real^n} \Big|\, |x|^\beta f\Big|^p\,dx \bigg]^{1/p} 
\end{equation} 
with the index constraint
$$\frac{n}p + \frac{n}q + \beta-\alpha = n\ .$$
\end{pitt}

For $\alpha =\beta$, $p$ and $q$ are dual exponents. 
Changing notation in equation~\eqref{thm3-eq}, one now has the following 
form of Pitt's inequality:

\begin{thm}\label{thm-pitt}
For $f\in \S(\real^n)$ with $0<\beta <1$, $1<p\le 2$ and $\frac1p +\frac1{p'} =1$
\begin{gather*}
\bigg[ \int_{\real^n} \Big|\, |x|^{-\beta} \widehat f\, \Big|^{p'}\,dx\bigg]^{1/p'} 
\le A\bigg[ \int_{\real^n} \Big|\, |x|^\beta f\Big|^p \,dx\bigg]^{1/p}\\
\noalign{\vskip6pt}
A =
 \left[ p^{1/p} \,\big/\, p'{}^{1/p'}\right]^{n/2} 
\frac{\left[ \int_{\real^n} |e^{2\pi i w\cdot\eta} - 1|^{p'} \frac1{|w|^{n+p'\beta}}\, dw\right]^{1/p'}}
{\left[ \int_{\real^n} |1-|x|^{-\lambda} |^{p'} |x-\eta|^{-n-p'\beta}  \, dx \right]^{1/p'}} 
\end{gather*}
with $\lambda = (n-p'\beta)/p'$.
\end{thm}

Pitt's inequality is a natural expression of the uncertainty principle --- to measure the 
balance between the relative size of a function and its Fourier transform at infinity. 
This underlying structure was outlined in an earlier paper \cite{Beckner-95}. 
Not only does Pitt's inequality determine uncertainty, but it offers insight into the nature 
of optimal constants and provides an asymptotic upper bound for the constant given 
by the extended uncertainty inequality:
\begin{equation}\label{eq:extended}
\bigg[ \int_{\real^n} |f|^2 \,dx\bigg]^2 
\le B_\alpha \int_{\real^n} |x|^\alpha |f|^2\,dx 
\int_{\real^n} |\xi|^\alpha |\widehat f\,|^2\,d\xi
\end{equation}
with $B_\alpha \lesssim (4\pi/n)^\alpha$ for $\alpha >0$. 
Observe that if the non-optimal constant taken from Pitt's inequality is used for $B_\alpha$
$$\pi^\alpha \left[ \Gamma \Big(\frac{n-\alpha}4\Big)\,\Big/ \Gamma \Big(\frac{n+\alpha}4
\Big)\right]^2$$
then one obtains the sharp logarithmic uncertainty inequality in the limit $\alpha\to0$.

\section*{Appendix: Proof of the symmetrization lemma}

The ``two-point'' inequality lies at the heart of many results on rearrangement and 
symmetrization echoing a central spirit from the work of Hardy and Littlewood.
Here the simplest form is given by the numerical inequality for sets of distinct positive 
real numbers $\{a_1,a_2\}$, $\{b_1,b_2\}$ with 
$$a_1 b_2 + a_2 b_2 \le a^* b^* + a_* b_*$$
where $c^* = \max \{c_1,c_2\}$ and $c_* = \min \{c_1,c_2\}$. 
Extending to finite sequences  $\{c_k\}$ with $\{c_k^*\}$ given by sequential 
rearrangement in terms of decreasing size 
$$\sum_1^N a_k b_k \le \sum_1^N a_k^* b_k^*\ .$$
The Symmetrization Lemma used in the argument for Theorem~\ref{frac-powers1} above 
corresponds to a two-function rearrangement inequality and its proof depends on two 
simple ideas which are outwardly independent of the geometric character of 
Lebesgue measure:
\begin{itemize}
\item[(1)] symmetrization of functions or sets can be achieved as the limit of a sequence 
of lower-dimensional symmetrizations; 
\item[(2)] rearrangement inequalities involving only two functions can often be 
obtained from symmetrization on two points.
\end{itemize}
The argument depends on having a manifold with a reflection symmetry that divides the 
space into two equivalent half-spaces with the property that for two points $P,Q$ in the 
same half-space then 
$$\text{dist} (P,Q) \le \text{dist} (P,\tilde Q)$$
where $\tilde Q$ is the reflection of $Q$ into the other half-space. 
In addition, one needs a transitive group action that ``moves points on the manifold 
around''. 

\begin{two-pointLem}
Consider a $\sigma$-finite measure space $M$ invariant under an involution 
symmetry $\sigma$. 
Suppose $M$ has a mutually disjoint decomposition $M = M_+ \cup M - \cup P$ with 
$P = \{x\in M :\sigma (x)=x\}$ being a set of measure zero, $M_- = \sigma (M_+)$ and 
$d(x,y) \le d(x,\sigma (y))$ for $s,y\in M_+$. 
$k$ and $\rho$ are two nonnegative functions defined on $M\times M$ with 
\begin{itemize}
\item[(i)] $k\big[ \sigma (x),\sigma (y)\big] = k(x,y)$, $\rho \big[ \sigma (x),\sigma (y)\big]
= \rho(x,y)$
\item[(ii)] $k(x,y) \ge k\big[ x,\sigma (y)\big]$, $\rho (x,y) \le \rho \big[ x,\sigma (y)\big]$ 
for $x,y\in M_+$.
\end{itemize}
$\varphi$ is a nonnegative function defined on $[0,\infty)$ with the following properties:
$\varphi (0)=0$, $\varphi$ convex and monotone increasing, $\varphi''(0)\ge0$ 
and $f\varphi' (t)$ convex. 
Define $f^* (x) =\max \{f(x),f[\sigma (x)]\}$ and 
$f^* [\sigma (x)]= \min \{f(x),f[\sigma (x)]\}$ for $x\in M_+$ and functions defined 
everywhere on $M$.
Then 
\begin{equation}\label{eq:2point}
\int_{M\times M} \varphi \left[ \frac{|f^*(x) - g^*(y)|}{\rho (x,y)}\right] 
k(x,y)\,dx\,dy 
\le \int_{M\times M} \varphi \left[ \frac{|f(x)-g(y)|}{\rho (x,y)}\right] 
k(x,y)\, dx\,dy
\end{equation}
\end{two-pointLem}
\bigskip

\noindent
The key to the proof of this two-point inequality lies with two simple observations. 

\begin{lem}\label{lem-A2}
For $\varphi$ convex with $\varphi (0)=0$ and nonnegative numerical sequences 
$\{a_1,a_2\}$ and $\{b_1,b_2\}$
\begin{equation}\label{eq:lem-A2}
\varphi \big[ |a_1 -b_1|\big] + \varphi \big[ |a_2 -b_2|\big] 
\ge \varphi \big[ |a_1^* - b_1^*|\big] + \varphi \big[ |a_2^* - b_2^*|\big]
\end{equation}
\end{lem}

\begin{proof} 
This result is determined by 
applying the property of convex functions that slopes are increasing; that is, 
$\varphi (s+t) - \varphi (t)$ is increasing in $t$.
\end{proof}

\begin{lem}\label{lem-A3} 
For $\varphi$ convex, $\varphi (0)=0$ and $t\varphi' (t)$ convex, then 
\begin{equation}\label{eq:lem-A3} 
T(\lambda) = \varphi \big[ \lambda |a_1 -b_1|\big] 
 + \varphi \big[ \lambda |a_2 - b_2|\big] 
 - \varphi \big[ \lambda |a_1^* - b_1^*|\big] 
 - \varphi \big[ \lambda |a_2^* - b_2^*|\big]
 \end{equation}
 is nondecreasing for $\lambda >0$.
 \end{lem}
 
 \begin{proof} 
 Observe that $t\varphi' (t)$ convex implies that $T' (\lambda) \ge 0$ which determines 
 that not only does $\varphi (\lambda |a_1 -b_1|) + \varphi (\lambda |a_2 -b_2|)$ 
 not increase under rearrangement but that the variation does not decrease under scaling
 increase.
 \renewcommand{\qed}{}
 \end{proof}
 
 \begin{proof} 
 \begin{equation*}
 \begin{split} 
 &\int_{M\times M} \varphi \left[ \frac{|f(x) - g(y)|}{\rho (x,y)}\right] k(x,y)\,dx\,dy\\
 \noalign{\vskip6pt}
 &\qquad 
 = \int_{M_+\times M_+} \left[ \left\{ \varphi 
 \left[ \frac{f(x) - g(y)}{\rho (x,y)}\right]  
 + \varphi \left[ \frac{|f(\sigma (x)) - g(\sigma (y))|}{\rho (x,y)}\right] \right\} k(x,y) \right.\\
  \noalign{\vskip6pt}
 &\qquad \qquad 
\left.  +\ \left\{\varphi \left[ \frac{|f(\sigma (x)) - g(y)|}{\rho (x,\sigma(y))}\right] 
+ \varphi \left[ \frac{|f(x) - g(\sigma (y))}{\rho (x,\sigma(y))} \right] \right\} k(x,\sigma (y))
\right]\,dx\,dy\\
\noalign{\vskip6pt}
&\qquad 
= \int_{M_+\times M_+} \left\{ \varphi \left[ \frac{|f(x) - g(y)|}{\rho (x,y)}\right]
+ \varphi \left[ \frac{|f(\sigma (x)) - g(\sigma (y))|}{\rho (x,y)} \right] \right\} 
\big[ k(x,y) - (x,\sigma (y)\big]\,dx\,dy \\
\noalign{\vskip6pt}
&\qquad\qquad 
+ \int_{M_+\times M_+} \left\{ \varphi \left[ \frac{|f(x) - g(y)|}{\rho (x,y)}\right] 
+ \varphi \left[ \frac{|f(\sigma (x)) - g(\sigma(y))|}{\rho (x,y)} \right] 
- \varphi \left[ \frac{|f(x) - g(y)|}{\rho (x,\sigma (y))}\right] \right.\\
\noalign{\vskip6pt}
&\qquad\qquad \qquad
\left. -\ \varphi \left[ \frac{|f(\sigma (x)) - g(\sigma (y))|}{\rho (x,\sigma (y))}\right] \right\}\, 
k(x,\sigma (y))\,dx\,dy\\
\noalign{\vskip6pt}
&\qquad \qquad
+ \int_{M_+ \times M_+} \left[ \left\{ \varphi \left[ \frac{|f(x) - g(y)|}{\rho (x,\sigma (y))}\right]
+ \varphi \left[ \frac{|f(\sigma (x)) - g(\sigma (y))|}{\rho (x,\sigma (y))}\right] 
+ \varphi \left[ \frac{|f(\sigma (x)) - g(y)|}{\rho (x,\sigma (y))}\right] \right. \right.\\
\noalign{\vskip6pt}
&\qquad\qquad \qquad
\left. +\ \varphi \left[ \frac{|f(x) - g(\sigma (y))|}{\rho (x,\sigma (y))}\right] \right\} \, 
k(x,\sigma (y))\,dx\,dy
 \end{split}
 \end{equation*}
 Since $\varphi$ is convex, the first integral decreases when $f$ and $g$ are replaced 
 by $f^*$ and $g^*$ using Lemma~\ref{lem-A2}; 
 since $t\varphi' (t)$ is convex, the second integral decreases when $f$ and $g$ are 
 replaced by $f^*$ and $g^*$ using Lemma~\ref{lem-A3}. 
 Finally the third integral is invariant under symmetrization. 
 Hence the proof of the Two-Point Symmetrization Lemma is complete by using 
 simple point-wise estimates for the integrand.
 \end{proof}
 
 \begin{cor}[strict monotonicity]
 Assume that $k(x,y) > k[x,\sigma (y)]$ for $x,y\in M_+$, $\varphi$ is strictly convex, 
 and there exist sets of positive measure $A$ and $B$ in $M_+$ with $f(x) > f[\sigma (x)]$ 
 for $x\in A$ and $g(x) < g[\sigma (x)]$ for $x\in B$. 
 Then inequality~\eqref{eq:2point} becomes a strict inequality.
 \end{cor}
 
 To complete the proof of the Symmetrization Lemma for the classical manifolds
 $\real^n$, $S^n$ and $\HH^n$ (hyperbolic space), the two curved spaces can be 
 represented as sitting in $\real^{n+1}$ with a preferred point of reference, and a 
 Euclidean half-space is specified so that reflection across the half-space has the 
 required property for the metric --- namely, reflection symmetry increases distance 
 between points in the sense that 
 $$\text{dist}(x,y) \le \text{dist}(x,\sigma (y))$$
 for two points $x$ and $y$ in the specified half-space.
 
 \begin{SymmLem}
 Let $M$ be a geometric manifold that possesses:
 {\rm 1)}~a transitive group action under which the metric is invariant, and 
 {\rm 2)}~reflection symmetry as defined above.
 Equimeasurable radial decreasing rearrangement is defined in terms of geodesic distance.
 Let $\varphi,K,\rho$ be nonnegative functions defined on $[0,\infty)$ with the following 
 properties:
 {\rm (i)}~$\varphi (0) =0$, $\varphi$ convex and monotone increasing, $\varphi'' \ge 0$,
 and $t\varphi'(t)$ convex, 
 {\rm (ii)}~$K$ monotone decreasing, and 
 {\rm (iii)}~$\rho$ monotone increasing; and $d(x,y)$ is the distance between $x$ and $y$.
 Then for measurable functions $f$ and $g$
 \begin{equation}\label{eq:symlem1}
 \begin{split}
& \int_{M\times M}\varphi \left[ \frac{|f(x)-f(y)|}{\rho [d(x,y)]}\right] K\big[ d(x,y)\big]\,dx\,dy\\
\noalign{\vskip6pt}
&\qquad
 \ge \int_{M\times M} \varphi \left[ \frac{|f^*(x)-f^*(y)|}{\rho [d(x,y)]}\right] K\big[ d(x,y)\big]\,dx\,dy
 \end{split}
 \end{equation}
 \end{SymmLem}
 
 If $K$ is strictly decreasing and $\varphi$ is strictly convex, then strict inequality holds 
 unless $f(x) = \lambda f^* (\tau x)$ and $g(x) = \lambda g^* (\tau x)$ with $|\lambda|=1$ 
 and $\tau x$ a translate of $x$, or in the case of a finite measure one of the functions 
 $f,g$ is constant almost everywhere. 
 In the case $\rho$ is constant, the last condition in the hypothesis on $\varphi$ can be 
 dropped. 
 
 \begin{proof} 
 Fix an origin in $M$ and choose a ``hyperplane'' that does not pass through this point. 
 $M_+$ will denote the half-space containing the origin and $\sigma$ will be reflection 
 through this hyperplane.
 By possibly changing values on a set of measure zero, $f$ and $g$ will be defined 
 everywhere on $M$.
 Choose a sequence of two-point symmetrizations which when applied to $f,g$ give 
 sequences $f_n,g_n$ that converge almost everywhere to the radial decreasing 
 rearrangements $f^*,g^*$ on $M$.
 Each two-point symmetrization applied to a function gives a rearrangement of that 
 function.
 The ``two-point'' lemma shows that 
 $$\int_{M\times M} \varphi \left[ \frac{|f_n(x)-g_n(y)|}{\rho [d(x,y)]}\right] K\big[ d(x,y)\big]\,dx\,dy
 $$
 is a decreasing sequence of numbers as $n\to\infty$. 
 Applying Fatou's lemma to this sequence of non-negative $L^1 (M\times M)$ functions 
 gives inequality~\eqref{eq:symlem1}. 
 If $f$ and $g$ are not almost everywhere radial decreasing functions with  respect 
 to the same fixed point on $M$, then the {\em first\/} two-point symmetrization can be 
 chosen to give a strict decrease in the inequality.
 This fact is equivalent to the existence of a half-space $M_+$ such that the set 
 $$\left\{ (x,y)\in M_+\times M_+ :   \big[ f(x) - f\big(\sigma (x)\big)\big]
 \big[ g(y) - g\big(\sigma (y)\big)\big] < 0\right\}$$
 has positive measure.
 \end{proof}

\section*{Acknowledgements}

I would like to thank Nestor Guillen for drawing my interest to kinetic problems in Landau 
collision dynamics. 
In addition, I would like to thank Guozhen Lu for arranging my visit to Nanjing University 
where some parts of this paper were developed, and 
for drawing my attention to his recent work with Han and Zhu (\cite{HLZ}). 
Paul Garrett's notes on ``Classical homogeneous spaces'' provided useful background.


\end{document}